\documentclass[11 pt]{amsart}
\usepackage{eurosym}
\usepackage{amsfonts}
\usepackage{amssymb}
\usepackage{mathrsfs}

\newtheorem{theorem}{Theorem}[subsection]
\newtheorem{corollary}[theorem]{Corollary}
\newtheorem{lemma}[theorem]{Lemma}

\theoremstyle{definition}
\newtheorem{definition}[theorem]{Definition}
\theoremstyle{remark}
\newtheorem{remark}[theorem]{Remark}
\numberwithin{equation}{subsection}
\newtheorem{example}{Example}[section]

\begin{document}

\title{FIXED POINT RESULTS OF $\mathfrak{R}$-ENRICHED INTERPOLATIVE KANNAN
PAIR IN $\mathfrak{R}$-CONVEX METRIC SPACES}
\author{Mujahid Abbas}
\address{ University of Pretoria \newline
\indent Department of Mathematics and Applied
Mathematics \newline
\indent Hatfield 002, Pretoria, South Africa}
\email{mujahid.abbas@up.ac.za}

\author{Rizwan Anjum}
\address{ Riphah International University \newline
\indent Department of Mathematics Riphah Institute of Computing and
Applied Sciences \newline
\indent Lahore 54000, Pakistan}
\email{rizwan.anjum@riphah.edu.pk}

\author{Shakeela Riasat}
\address{ Government College
University \newline
\indent Abdus Salam School of Mathematical Sciences \newline
\indent Lahore 54600, Pakistan}
\email{shakeelariasat@sms.edu.pk}
\thanks{}
\thanks{}
\date{}
\dedicatory{}


\begin{abstract}
The purpose of this paper is to introduce the class of $\mathfrak{R}$-enriched interpolative Kannan pair and proved a common fixed point result
in the context of $\mathfrak{R}$-complete convex metric spaces. Some
examples are presented to support the concepts introduced herein. Moreover,
we study the well-posedness, limit shadowing property and Ulam-Hyers stability
of the mappings introduced herein. Our result extend and
generalize several comparable results in the existing literature.
\end{abstract}

\maketitle
\pagestyle{myheadings}
\markboth{M. Abbas, R. Anjum and S. Riasat }{ENRICHED INTERPOLATIVE KANNAN
PAIR IN}

\section{Introduction and Preliminaries}

The study of the existence and approximation of the solutions of nonlinear
functional equations such as differential equations, integral equations,
integro-differential equations is always been a source of great interest for
mathematicians. Fixed point theory plays a vital role in this aspect. The
problem of finding the solutions of a functional equation can be shifted to
finding the fixed point of a suitable mapping defined on a set endowed with
a certain structure, that is, finding the solution of the fixed point
equation given by
\begin{equation}
x=Tx,  \label{FPP}
\end{equation}%
where $T$ is defined on a certain space $X$. An equation (\ref{FPP}) is
solved by applying some appropriate fixed point results.\newline
Banach contraction mapping principle \cite{Banach} is one of the most useful
results for approximation of the solution of (\ref{FPP}), if a mapping $T: X \rightarrow X$ on a complete metric space $X$,
satisfies the following condition:
\begin{equation}
d(Tx,Ty)\leq cd(x,y),\ \ \text{for all }x,y\in X\   \label{BCP}
\end{equation}%
where $c\in \lbrack 0,1)$. The mapping
satisfying the above condition by taking $c=1$ is called nonexpansive.
\newline
Note that, a mapping which satisfies the above condition is uniformly
continuos mapping. Kannan \cite{Kannan} proved a fixed point result which
is applicable to the equation (\ref{FPP}) where the mapping defined on a
complete metric space need not be continuous. \ In this way, Kannan fixed
point result extends the scope of Banach contraction mapping principle \cite%
{Banach}. Kannan \cite{Kannan} replaced the condition (\ref{BCP}) with the
following:
\begin{equation}
d(Tx,Ty)\leq a\{d(x,Tx)+d(y,Ty)\},  \label{kc}
\end{equation}%
for all $x,y\in X,$ where $a\in \lbrack 0,0.5).$ It was shown that, an
equation (\ref{kc}) involving the Kannan mapping on a complete metric
space, has a unique solution. This paper was a genesis for a multitude of
fixed point papers over the next two decades, see \cite%
{b,Kannan1,Kannan,Kikkawa} and references mentioned therein.\newline
Note that Kannan fixed point result \cite{Kannan} is based on the
generalization of the condition (\ref{BCP}). The same applies to the most of
results in metric fixed point theory. One of such results proved by
Karapinar \cite{Karapinar}, who employed the technique of interpolation and introduced
a new class of interpolative Kannan type mappings defined on a complete
metric space. Let us recall the following definition:\newline
A mapping $T:X\rightarrow X$ is called an interpolative Kannan type if
there exists $a\in \lbrack 0,1)$ such that for all $x,y\in X\setminus
Fix(T), $ we \ have
\begin{equation}
d(Tx,Ty)\leq a[d(x,Tx)]^{\alpha }[d(y,Ty)]^{1-\alpha },  \label{IKC}
\end{equation}%
where $\alpha \in (0,1)$ and $Fix(T)=\{x\in X;\ \ x=Tx\}$\newline
The study of common fixed points of mappings satisfying certain contractive
conditions has been at the center of vigorous research activity. Common
fixed point results are useful in finding the simultaneous solution of \ a
certain system of two nonlinear functional equations. Noorwali \cite{N}
extended the condition (\ref{IKC}) to two mappings and proved a common
fixed point result for such mappings given as: Let $(X,d)$ be a complete metric space and $T,S: X \rightarrow X$. If there exists $a\in \lbrack 0,1)$ such that for all $x,y\in
X\setminus \{Fix(T),Fix(S)\},$ we \ have
\begin{equation*}
d(Tx,Sy)\leq a[d(x,Tx)]^{\alpha }[d(y,Sy)]^{1-\alpha },
\end{equation*}
where $\alpha \in (0,1).$ Then, $T$ and $S$ has a unique common fixed point in $X$.\newline
For more results in this direction, we refer to \cite%
{Agarwal,Aydi,Aydi2,Debnath,Gaba,Gaba2,Karapinar2,Karapinar3,Karapinar4,N})
and references mentioned therein.\newline

On the other hand, authors \cite{Abbas2,Abbas11,Abbas111,Anjum1,B6,B4,B2,main cite,B7,B9} used the technique of
enrichment of contractive mappings in the setting of Banach spaces. \newline
In 2021, Rizwan and Abbas \cite{AnjumNew} introduced the class of $(a,b,c)$-modified enriched Kannan pair $(T,S)$ in the setting of Banach space.
A pair of self mappings  $(T,S)$ is called $(a,b,0)$-modified enriched
Kannan pair on a normed space  $(X,\left\Vert \cdot \right\Vert )$, if there exist $b\in \lbrack 0,\infty )$ and $a\in
\lbrack 0,0.5)$ such that
\begin{equation*}
 \left\Vert b(x-y)+Tx-Sy \right\Vert\leq a\big\{\left\Vert x-Tx\right\Vert+\left\Vert y-Sy\right\Vert\big\},
\end{equation*}
holds for all $x,y\in X.$ It was proven that any $(a,b,0)$-modified enriched
Kannan pair  $(T,S)$ in the setting of Banach space $X$ admits a unique common fixed point.
\newline
An equation (\ref{FPP}) involving a nonexpansive mapping defined on a
complete metric space may have no solution or have infinitely many
solutions. A fundamental problem in fixed point theory of nonexpansive
mappings is to find conditions under which an equation (\ref{FPP}) has a
solution. It is intimately connected with differential equations and with
the geometry of the Banach spaces. The techniques used to prove the
existence of the solution of non expansive fixed point equations are
different than those considered in metric fixed point theory for mappings
satisfying certain contraction conditions. The interplay between the
geometry of Banach spaces and fixed point theory has been very strong and
fruitful. Geometrical properties of Banach spaces which mainly use convexity
hypothesis, play key role in solving an equation (\ref{FPP}) involving a
nonexpansive mapping. The results obtained in this direction were starting
point of a new mathematical field : the application of the geometric theory
of Banach spaces to fixed point theory.

In 1970, Takahashi \cite{Takahashi} introduced a notion of convexity
structure in a metric space with the aim of studying the fixed point problem
for nonexpansive mappings in such spaces.

\begin{definition}\rm\label{cpo}
\cite{Takahashi} Let $(X,d)$ be a metric space. A continuous function $%
W:X\times X\times \lbrack 0,1]\rightarrow X$ is said to be a convex
structure on $X$ if, for all $x,y\in X$ and for any $\lambda \in \lbrack
0,1] $, we have
\begin{equation*}
  d(u,W(x,y;\lambda ))\leq \lambda d(u,x)+(1-\lambda )d(u,y),\ \ \text{for any
}u\in X.
\end{equation*}
\end{definition}

A metric space $(X,d)$ endowed with a convex structure $W$ is called a
Takahashi convex metric space and is usually denoted by $(X,d,W)$.
Obviously, any linear normed space and each of its convex subsets are convex
metric spaces with the natural convex structure given by
\begin{equation*}
  W(x,y;\lambda )=\lambda x+(1-\lambda )y,\ \ \forall x,y\in X.
\end{equation*}
However, the converse does not hold in general. There are many examples of
convex metric spaces which cannot be embedded in any normed space. Let us
consider the following example.

\begin{example}\rm\cite{exp}
Let $X = \mathbb{R}^{2}$. For $x=(x_{1},x_{2})$, $y=(y_{1},y_{2})$ in $X$
and $\alpha \in [0,1]$. Define a mapping $W: X \times X \times [0,1]
\rightarrow X$ by
\begin{equation*}
W(x,y,\alpha)= \bigg(\alpha x_{1}+ (1-\alpha)y_{1}, \frac{\alpha x_{1}x_{2}
+ (1- \alpha)y_{1}y_{2}}{\alpha x_{1}+ (1-\alpha)y_{1}}\bigg),
\end{equation*}
and a metric $d: X \times X \rightarrow [0,\infty)$ by $d(x,y)= |x_{1}-
y_{1}|+ |x_{1}x_{2}- y_{1}y_{2}|$. It can be verified that $X$ is a convex
metric space but not a normed space.
\end{example}
   By using the above idea of Takahashi \cite{Takahashi}, Berinde and  P\u{a}curar \cite{main cite}, in 2021, proposed a new class of enriched contractions in a more generalized setting of convex metric spaces which is major inspiration for this  manuscript.\newline
We need the following lemma in the sequel.

\begin{lemma}
\cite{main cite} \label{fix} Let $(X,d,W)$ be a convex metric space and $%
T:X\rightarrow X$. Define the mapping $T_{\lambda }:X\rightarrow X$ by
\begin{equation}
T_{\lambda }x=W(x,Tx;\lambda ).  \label{2}
\end{equation}%
Then, for any $\lambda \in \lbrack 0,1),$ we have
\begin{equation*}
Fix(T)=Fix(T_{\lambda }).
\end{equation*}
\end{lemma}
\begin{remark}\rm\label{main rem}
  Since, convex structure $W$ is a continuous mapping on $X \times X \times [0,1]$, then by (\ref{2}), we can say that $T_{\lambda }$ is also a continuous mapping on $X$.
\end{remark}

It is a matter of great interest to study the necessary conditions which
ensure the existence and uniqueness of the solution of an equation (\ref{FPP}%
) where the mapping is defined on a space equipped with some order
structure in addition to the distance structure. Existence of fixed points
in partially ordered metric spaces was first investigated in 2004 by Ran and
Reurings \cite{ran} and then by Nieto and  L\'{o}pez \cite{nieto}. These results
weakened the underlying spaces and added a new dimension to metric fixed
point theory: the bridging between to metric and ordered fixed point
theories. \newline
In 2017, Eshaghi et al. \cite{ortho}, gave the idea of orthogonal relation
on metric spaces which was generalized by Rahimi et al. \cite{R metric}, in
2020, by introducing the concept of $\mathfrak{R}$-metric spaces, where $\mathfrak{R}$ is an arbitrary relation on a metric space $X$. \newline
We now give the notion of $\mathfrak{R}$-convex metric space as follows:

\begin{definition}\rm
Suppose $(X,d,W)$ is a convex metric space and $\mathfrak{R}$ is an
arbitrary relation on $X$. Then $(X,d,W,\mathfrak{R})$ is called an $\mathfrak{R}$-convex metric space.
\end{definition}

\begin{example}\rm
Let $X=\mathbb{R}\times \mathbb{R}$, $a=(\alpha ,\beta )$, and $b=(\zeta
,\eta )$ be two arbitrary points in $X$. Define a relation $\mathfrak{R}%
:=\leq $ on $X$ as $a\leq b$ if and only if $\alpha \leq \zeta $ and $\beta
\leq \eta $. Let $d:X\times X\rightarrow \mathbb{R}$ be the metric given by
\begin{equation*}
d(a,b)=\left\{
\begin{array}{cc}
\sqrt{(\alpha -\zeta )^{2}+(\beta -\eta )^{2}} & ;if\ \ a\mathfrak{R}b \\
\max \{|\alpha -\zeta |,|\beta -\eta |\} & ;otherwise.%
\end{array}%
\right.
\end{equation*}
Convex structure $W$ on $X$ is given by $W(a,b;\lambda )=(\lambda \alpha
+(1-\lambda )\zeta ,\lambda \beta +(1-\lambda )\eta )$. Then one can easily
see that $X$ is an $\mathfrak{R}$-convex metric space.
\end{example}

\begin{definition}\rm
A sequence $\{x_{n}\}$ in an $\mathfrak{R}$-convex metric space is called an
$\mathfrak{R}$-sequence if $x_{n}\mathfrak{R}x_{n+k}$ for each $n,k \in
\mathbb{N}$.
\end{definition}

\begin{definition}\rm
An $\mathfrak{R}$-sequence $\{x_{n}\}$ is said to converges to $x\in X$ if
for every $\epsilon >0$ and $n\in \mathbb{N,}$ there is an integer $n_{0}$
such that $d(x_{n},x)<\epsilon $ for $n\geq n_{0}$.
\end{definition}

\begin{definition}\rm
An $\mathfrak{R}$-sequence $\{x_{n}\}$ is said to be an $\mathfrak{R}$%
-Cauchy sequence if for every $\epsilon >0$ and $m,n\in \mathbb{N}$, there
is an integer $n_{0}$ such that $d(x_{n},x_{m})\leq \epsilon $ for $n,m\geq
n_{0}$.
\end{definition}

\begin{definition}\rm
Let $(X,d,W,\mathfrak{R})$ be an $\mathfrak{R}$-convex metric space. Then it
is called an $\mathfrak{R}$-complete convex metric space if every $\mathfrak{R}$-Cauchy sequence converges in $X$.
\end{definition}

Employing the ideas of Berinde and  P\u{a}curar \cite{main cite}, Rizwan and Abbas \cite{AnjumNew}, Karapinar \cite{Karapinar} and Rahimi \cite{R metric}, we introduce the concept of $\mathfrak{R}$-enriched interpolative Kannan pair and prove some fixed point results in
the setting of $\mathfrak{R}$-complete convex metric spaces. Moreover, we
study the well-posedness, limit shadowing property and Ulam-Hyers stability
of the mappings introduced herein.
\section{Main Results}

We first introduce the following concept.

\begin{definition}\rm
Let $(X,d,W,\mathfrak{R})$ be an $\mathfrak{R}$-convex metric space. A pair of mappings $%
T,S:X\rightarrow X$ is said to be enriched interpolative Kannan
pair if there exist  $a,\lambda\in \lbrack 0,1)$ and $\alpha \in (0,1)$ such that
\begin{equation}
d(W(x,Tx;\lambda ),W(y,Sy;\lambda ))\leq a[d(x,W(x,Tx;\lambda ))]^{\alpha
}[d(y,W(y,Sy;\lambda ))]^{1-\alpha },  \label{main def}
\end{equation}%
holds for all $x\mathfrak{R}y$, where $x,y\in X$.\newline
To specify the parameters $a,\lambda$ and $\alpha$ and relation $\mathfrak{R}$ in (\ref{main def}), we also call  $(T,S)$ a $(a,\alpha ,\lambda ,\mathfrak{R})$-enriched interpolative Kannan pair.
\end{definition}
We start with the following common fixed point result.

\begin{theorem}
\label{main thm} Suppose that $(X,d,W,\mathfrak{R})$ is an $\mathfrak{R}$-complete convex metric space and $(T,S)$ is  $(a,\alpha ,\lambda,\mathfrak{R})$-enriched interpolative Kannan pair.
Suppose the following conditions hold:
\begin{enumerate}
  \item there exists $x_{0}\in X$ such that $x_{0}\mathfrak{R}W(x_{0},Tx_{0};\lambda),$
  \item if for all $x,y\in X$ such that $x\mathfrak{R}y,$
  \begin{equation}\label{bgt}
\Rightarrow \ \ \ \    W(x,Tx;\lambda) \mathfrak{R} W(y,Sy;\lambda) \ \ \text{or} \ \     W(x,Sx;\lambda) \mathfrak{R} W(y,Ty;\lambda).
  \end{equation}
\end{enumerate}
Then $S$ and $T$ have common fixed point $p\in X.$ Moreover, the iterative sequence $%
\{x_{n}\}_{n=0}^{\infty }$ defined by
\begin{equation}
x_{2n+1}=W(x_{2n},Tx_{2n};\lambda ), \ \ \ x_{2n+2}=W(x_{2n+1},Sx_{2n+1};\lambda )\ \ n\geq 0, \label{001}
\end{equation}%
converges to the point $p.$
\end{theorem}

\begin{proof}
Using Lemma \ref{fix}, the condition (\ref{main def}) is transformed to
following equivalent form
\begin{equation}
d(T_{\lambda }x,S_{\lambda }y)\leq a[d(x,T_{\lambda }x)]^{\alpha
}[d(y,S_{\lambda }y)]^{1-\alpha },  \label{e form}
\end{equation}%
holds for all $x\mathfrak{R}y.$ That is, a pair $(T_{\lambda },S_{\lambda })$ becomes an interpolative
Kannan pair. Note that the equation (\ref{001}) becomes
\begin{align*}
x_{2n+1}& =T_{\lambda }x_{2n}  \\
x_{2n+2}& =S_{\lambda }x_{2n+1},\ \ n\geq 0.
\end{align*}%
Without any loss of generality, we assume that the successive terms of $%
\{x_{n}\}$ are distinct. Otherwise, we are done. Since, $x_{0}\mathfrak{R} T_{\lambda}x_{0}=x_{1}$  by (\ref{bgt}), we have
\begin{eqnarray*}
(x_{1}=T_{\lambda}x_{0})\mathfrak{R} (S_{\lambda}x_{1}=x_{2}).
\end{eqnarray*}%
Continuing this way, we obtain that
\begin{equation*}
x_{1}\mathfrak{R}x_{2}\mathfrak{R}\dotsc \mathfrak{R}x_{n}\mathfrak{R}x_{n+1}%
\mathfrak{R}\dotsc\text{.}
\end{equation*}%
Thus $\{x_{n}\}_{n=0}^{\infty}$ is an $\mathfrak{R}$-sequence in $(X,d,W,\mathfrak{R})$.%
\newline
Take $x=x_{2n}$ and $y=x_{2n+1}$ in (\ref{e form}), we get
\begin{align}
d(x_{2n+1},x_{2n+2})& =d(T_{\lambda }x_{2n},S_{\lambda }x_{2n+1})  \notag
\label{03} \\
& \leq a[d(x_{2n},T_{\lambda }x_{2n})]^{\alpha }[d(x_{2n+1},S_{\lambda
}x_{2n+1})]^{1-\alpha }  \notag \\
& \leq a[d(x_{2n},x_{2n+1})]^{\alpha }[d(x_{2n+1},x_{2n+2})]^{1-\alpha }
\notag \\
\lbrack d(x_{2n+1},x_{2n+2})]^{\alpha }& \leq a[d(x_{2n},x_{2n+1})]^{\alpha }
\notag \\
& \leq ad(x_{2n},x_{2n+1})  \notag \\
& \leq a^{2n+1}d(x_{0},x_{1}).
\end{align}%
Similarly, for $x=x_{2n}$ and $y=x_{2n-1}$ in (\ref{e form}), we obtain
that
\begin{align}
d(x_{2n+1},x_{2n})& =d(T_{\lambda }x_{2n},S_{\lambda }x_{2n-1})  \notag
\label{04} \\
& \leq a[d(x_{2n},T_{\lambda }x_{2n})]^{\alpha }[d(x_{2n-1},S_{\lambda
}x_{2n-1})]^{1-\alpha }  \notag \\
& \leq a[d(x_{2n},x_{2n+1})]^{\alpha }[d(x_{2n-1},x_{2n})]^{1-\alpha }
\notag \\
\lbrack d(x_{2n+1},x_{2n})]^{1-\alpha }& \leq
a[d(x_{2n},x_{2n-1})]^{1-\alpha }  \notag \\
& \leq ad(x_{2n},x_{2n-1})  \notag \\
& \leq a^{2n}d(x_{0},x_{1}).
\end{align}%
It follows from(\ref{03}) and (\ref{04}) that
\begin{equation*}
d(x_{n},x_{n+1})\leq a^{n}d(x_{0},x_{1}).
\end{equation*}%
For $r>0$, we have
\begin{align*}
d(x_{n},x_{n+r})& \leq
d(x_{n},x_{n+1})+d(x_{n+1},x_{n+2})+d(x_{n+2},x_{n+3})+\dotsc+
\\
& \ \ \ \ \ \ d(x_{n+r-1},x_{n+r})\\
& \leq (a^{n}+a^{n+1}+a^{n+2}+\dotsc+a^{n+r-1})d(x_{0},x_{1}) \\
& \leq \frac{a^{n}}{1-a}d(x_{0},x_{1}).
\end{align*}%
On taking limit as $n\rightarrow \infty $, we can deduce that $\{x_{n}\}_{n=0}^{\infty}$ is
an $\mathfrak{R}$-Cauchy sequence. Since $X$ is $\mathfrak{R}$-complete,
there exists $p\in X$ such that $\lim_{n\rightarrow \infty }x_{n}=p$ which implies
\begin{equation}\label{qwer}
  \lim_{n\rightarrow \infty }x_{2n}=p.
\end{equation}
By using (\ref{qwer}) and the $\mathfrak{R}$-continuity of  $T_{\lambda}$   (which follows from Remark \ref{main rem} and Definition \ref{cpo}), we immediately obtain
\begin{equation*}
  \lim_{n\rightarrow \infty }T_{\lambda}x_{2n}=T_{\lambda}p.
\end{equation*}
Therefore,
\begin{equation*}
 T_{\lambda}p=T_{\lambda}(\lim_{n\rightarrow \infty }x_{2n})=\lim_{n\rightarrow \infty }T_{\lambda}(x_{2n})=\lim_{n\rightarrow \infty }x_{2n+1}=p,
\end{equation*}
i.e.,
\begin{equation*}
  p\in Fix(T_{\lambda})=Fix(T).
\end{equation*}
 On the similar arguments $p$ is the fixed point of $S$ as well.   \newline

\end{proof}
 For more illustration, consider the following examples.

\begin{example}\rm
Let $X=\mathbb{R}^{2}$ and $D=\{(x,y)\in \mathbb{R}^{2}:\ \ x=y\}$. For $x=(x_{1},x_{2})$, $y=(y_{1},y_{2})$ in $X$ and $\alpha
\in \lbrack 0,1]$. Define a mapping $W:X\times X\times \lbrack
0,1]\rightarrow X$ by
\begin{equation*}
W(x,y,\alpha )=(\alpha x_{1}+(1-\alpha )y_{1},\alpha x_{2}+(1-\alpha )y_{2}),
\end{equation*}%
and a metric $d:X\times X\rightarrow \lbrack 0,\infty )$ by $%
d(x,y)=|x_{1}-y_{1}|+|x_{2}-y_{2}|$. \newline
Define a relation $\mathfrak{R}$ on $X$ as $x\mathfrak{R}y$ if and only if $%
x,y\in D$. It can be easily seen that the $(X,d,W,\mathfrak{R})$ is an $\mathfrak{R}$-complete convex metric
space.  Define the mappings $T,S:X\rightarrow X$ by
\begin{equation*}
T(x)=%
\begin{cases}
(x_{1},2x_{1}-x_{2}) & \text{if $x=(x_{1},x_{2})\notin D$} \\
(-x_{1},-x_{2}) & \text{if $x=(x_{1},x_{2})\in D,$}%
\end{cases}%
\end{equation*}
and
\begin{equation*}
S(x)=%
\begin{cases}
(6x_{2}-x_{1},5x_{2}) & \text{if $x=(x_{1},x_{2})\notin D$} \\
(-x_{1},-x_{2}) & \text{if $x=(x_{1},x_{2})\in D.$}
\end{cases}%
\end{equation*}
On the other hand, we have
\begin{equation}\label{234e}
W(x,Tx;0.5)=%
\begin{cases}
(x_{1},x_{1}) & \text{if $x=(x_{1},x_{2})\notin D$} \\
(0,0) & \text{if $x=(x_{1},x_{2})\in D,$}
\end{cases}%
\end{equation}
and
\begin{equation}\label{xe}
W(x,Sx;0.5)=%
\begin{cases}
(3x_{2},3x_{2}) & \text{if $x=(x_{1},x_{2})\notin D$} \\
(0,0) & \text{if $x=(x_{1},x_{2})\in D.$}
\end{cases}%
\end{equation}
If we choose $x_{0}=(0,0),$ then
\begin{equation*}
  (0,0)\mathfrak{R}W(x_{0},Tx_{0},0.5)=(0,0).
\end{equation*}
By using (\ref{234e}), (\ref{xe}) and $x\mathfrak{R}y$ we have
\begin{eqnarray*}
(0,0)=W(x,Tx,0.5)&\mathfrak{R}&W(y,Sy,0.5)=(0,0), \text{ or} \\
(0,0)=W(x,Sx,0.5)&\mathfrak{R}&W(y,Ty,0.5)=(0,0).
\end{eqnarray*}%
As a result,  the  assumptions i) and ii) of Theorem \ref{main thm} hold.

Note that $(T,S)$ is an $(0.5,0.5 ,0.5,\mathfrak{R})$-enriched
interpolative Kannan pair.\newline
Indeed, for $x=(x_{1},x_{2}) \mathfrak{R} y=(y_{1},y_{2}),$ we have
\begin{eqnarray*}
  d (W(x,Tx;0.5),W(y,Sy;0.5))&=& 0,  \\
  d(x,W(x,Tx;0.5))&=& 2|x|, \text{ and} \\
  d(y,W(y,Sy;0.5))&=& 2|y|.
\end{eqnarray*}%
Therefore, for all $x\mathfrak{R}Sy$ we have
\begin{align*}
d(W(x,Tx;\lambda ),W(y,Sy;\lambda ))&=0    \\
   & \leq a[d(x,W(x,Tx;\lambda ))]^{0.5
}[d(y,W(y,Sy;\lambda ))]^{0.5 }\\
&=0.5(2|x|)^{0.5}(2|y|)^{0.5}.
\end{align*}
Thus, all the conditions of Theorem \ref{main thm}
are satisfied. Hence, $x=(0,0)$ is the common fixed point of $T$ and $S.$
\end{example}

\begin{corollary}\label{application}
  Let $(X,d,W)$ be a complete convex metric space and $(T,S)$ be $(a, \alpha, \lambda)$-enriched interpolative Kannan pair. Then the pair $(T,S)$ has a unique common fixed point.
\end{corollary}
\begin{proof}
  Define the relation $\mathfrak{R}$ on $X$ as $\mathfrak{R}= X\times X$ in Theorem \ref{main thm}, then the result follows from Theorem \ref{main thm} . To prove the uniqueness, let $p$ and $q$ be two distinct common fixed points of the mappings $T$ and $S$ and by Lemma \ref{fix}, also of $T_{\lambda}$ and $S_{\lambda}$, then by (\ref{e form})
\begin{equation*}
  d(p,q)=d(T_{\lambda}p,S_{\lambda}q)\leq a[d(p,T_{\lambda}p)]^{\alpha}[d(q,S_{\lambda}q)]^{1-\alpha}=0
\end{equation*}
Hence, $p=q$.
\end{proof}
\begin{corollary}
\label{cor1} Let $(X,\left\Vert \cdot \right\Vert )$ be Banach space and $T,S:X\rightarrow X$  be a mappings satisfying
\begin{equation}\label{rizwananju}
\left\Vert b(x-y)+(Tx-Sy)\right\Vert \leq a\left\Vert x-Tx\right\Vert
^{\alpha }\left\Vert y-Sy\right\Vert ^{1-\alpha },
\end{equation}%
is satisfied for all  $x,y\in X$ such that $Tx\neq x$ whenever $Sy\neq y,$ with $b\in[0,\infty),$ $a\in[0,1)$ and $\alpha\in(0,1).$ Then $S$ and $T$ have a unique common fixed point.
\end{corollary}
\begin{proof}
     We choose $\lambda=\frac{1}{b+1},$ then condition (\ref{rizwananju}) becomes
     \begin{equation*}
       \left\Vert (1-\lambda )(x-y)+\lambda (Tx-Sy)\right\Vert \leq a\lambda
\left\Vert x-Tx\right\Vert ^{\alpha }\left\Vert y-Sy\right\Vert ^{1-\alpha } \ \ \ \forall x,y\in X.
     \end{equation*}
     which can be written in an equivalent form as:
     \begin{equation}\label{bgtre}
         \left\Vert T_{\lambda}x-S_{\lambda}y\right\Vert \leq a\left\Vert x-T_{\lambda}x\right\Vert
^{\alpha }\left\Vert y-S_{\lambda}y\right\Vert ^{1-\alpha }, \ \ \ \forall x,y\in X.
     \end{equation}
 If we define convex structure $W$ as $W(x,y;\alpha)=(1-\alpha)x+\alpha y $ on a Banach space $X$ equipped with a relation $\mathfrak{R}=X\times X.$ Since $X$ is Banach space, it is $\mathfrak{R}$-Banach convex space.
The inequality (\ref{bgtre}) and we have that the mappings $S_{\lambda},T_{\lambda}:X\rightarrow X$
defined by (\ref{2})   suggest that $(T,S)$ a $(a,\alpha ,\lambda ,\mathfrak{R})$-enriched interpolative Kannan pair.
Corollary \ref{application} leads to the conclusion.

\end{proof}

As a corollary of our result (Corollary \ref{cor1}), we can obtain Theorem 2.1 of \cite{B6}.
\begin{corollary}\cite{B6}
  Let $(X,\left\Vert \cdot \right\Vert )$ be a Banach space and $%
T:X\rightarrow X$ be an $(b,a)$-enriched Kannan contraction, that is a
mapping satisfying:
\begin{equation} \label{EnrichedKannan}
||b(x-y)+Tx-Ty||\leq a\big\{ ||x-Tx||+||y-Ty||\big\} \ \forall \ x,y\in X,
\end{equation}
with $b\in [0,\infty)$ and $a\in[0,1/2).$ Then, $T$ has a unique fixed point.
\end{corollary}
\begin{proof}
 By taking $\lambda=\frac{1}{b+1},$ in (\ref{EnrichedKannan}) we get,
  \begin{equation*}
       \left\Vert (1-\lambda )(x-y)+\lambda (Tx-Ty)\right\Vert \leq a\lambda
\{\left\Vert x-Tx\right\Vert +\left\Vert y-Ty\right\Vert \} \ \ \ \forall x,y\in X,
     \end{equation*}
 which can be written in an equivalent form as:
     \begin{equation}
         \left\Vert T_{\lambda}x-T_{\lambda}y\right\Vert \leq a \{\left\Vert x-T_{\lambda}x\right\Vert +\left\Vert y-T_{\lambda}y\right\Vert \}, \ \ \ \forall x,y\in X.\label{222}
     \end{equation}
Therefore, by (\ref{222}), $T_{\lambda}$ is a Kannan contraction (\ref{kc}). It follows from \cite{N}, mapping $T_{\lambda}$ satisfying  (\ref{222}) also satisfies (\ref{bgtre}). Since for value $\lambda=\frac{1}{b+1},$ the inequality (\ref{bgtre}) is equivalent to (\ref{rizwananju}),  hence result follows from Corollary \ref{cor1}.

\end{proof}
As a corollary of our result (Corollary \ref{cor1}), we can obtain Theorem 2.1 of \cite{N}.
\begin{corollary}\cite{N} \label{cor2}
Let $(X,\left\Vert \cdot \right\Vert )$ be a Banach space and $T,S:X\rightarrow X$ be a mappings satisfying
\begin{equation*}
\left\Vert Tx-Sy\right\Vert \leq a\left\Vert x-Tx\right\Vert ^{\alpha
}\left\Vert y-Sy\right\Vert ^{1-\alpha},
\end{equation*}
 for all $ x,y\in X$ such that $Tx\neq x$ whenever $Ty\neq y,$
with  $a\in[0,1)$ and $\alpha\in(0,1).$
Then $S$ and $T$ have a unique common fixed point.
\end{corollary}
\begin{proof}
For $b=0,$  Corollary \ref{cor1} leads to the conclusion.
\end{proof}

\section{Well-Posedness, Limit Shadowing Property and Ulam-Hyers Stability}
Now we will present the well-posedness, limit shadowing property and Ulam-Hyers stability results for the Corollary \ref{application}.
\subsection{Well-Posedness}
Let us start with the following definition.

\begin{definition}\rm
Let $(X,d,W)$ be a convex metric space and $%
T,S:X\rightarrow X$. The common fixed point problem of a pair $(T,S)$ is
said to be well-posed if $Fix(T)=Fix(S)=\{p\}$ (say) and for any sequence $%
\{x_{n}\}_{n=0}^{\infty}$ in $X$ satisfying%
\begin{eqnarray*}
\lim_{n\rightarrow \infty }d(W(x_{n},Tx_{n},\lambda ),x_{n}) &=&0\text{ and }
\\
\lim_{n\rightarrow \infty }d(x_{n},W(x_{n},Sx_{n},\lambda )) &=&0,
\end{eqnarray*}
we have $\lim_{n\rightarrow \infty }x_{n}=p.$
\end{definition}

Since by Lemma \ref{fix} $Fix(T)=Fix(T_{\lambda })$, and $%
Fix(S)=Fix(S_{\lambda }),$ we conclude that the common fixed point problem
of $(T,S)$ is well-posed if and only if the fixed point problem of $%
(T_{\lambda },S_{\lambda })$ is well-posed.\newline
If we take $S=T$ in the above definition, we have a well-posedness of a
fixed point problem of a mapping $T.$ Well-posedness of certain fixed point
problems has been studied by several mathematicians, see for example, \cite%
{article.5}, \cite{article.11} and references mentioned therein.\newline
We now study the well-posedness of a common fixed point problem of mappings
in Theorem \ref{main thm}.

\begin{theorem}
\label{Well P} Let $(X,d,W)$ be a complete
convex metric space. Suppose that $T,S$ are mappings on $X$ as in the
Corollary \ref{application}. Then, common fixed point problem of a pair $(T,S)$
is well-posed.
\end{theorem}

\begin{proof}
It follows from Corollary \ref{application}, that $Fix(T)=Fix(S)=\{p\}.$ Let $%
\{x_{n}\}_{n=0}^{\infty} $ be any sequence in $X$ such that $\lim_{n\rightarrow \infty
}d(W(x_{n},Tx_{n},\lambda ),x_{n})=0$ and $\lim_{n\rightarrow \infty
}d(x_{n},W(x_{n},Sx_{n},\lambda ))=0.$  Then, from triangular inequality and (\ref{e form}), we have
\begin{align*}
d(x_{n},p)& \leq d(x_{n},W(x_{n},Sx_{n},\lambda ))+d(W(x_{n},Sx_{n},\lambda
),p) \\
& =d(x_{n},S_{\lambda }x_{n})+d(T_{\lambda }p,S_{\lambda }x_{n}) \\
& \leq d(x_{n},S_{\lambda }x_{n})+a[d(p,T_{\lambda }p)]^{\alpha
}[d(x_{n},S_{\lambda }x_{n})]^{1-\alpha }.
\end{align*}%
On taking limit as $n\rightarrow \infty $ on both sides of the above
inequality, we obtain that $d(x_{n},p)=0$. Hence the given common fixed
point problem of a pair $(T,S)$ is well-posed
\end{proof}

\subsection{Limit Shadowing Property}

\begin{definition}\rm
Let $(X,d,W)$ be a convex metric space and $%
T,S:X\rightarrow X$. The common fixed point problem of a pair $(T,S)$ is
said to have the limit shadowing property if for any sequence
$\{x_{n}\}_{n=0}^{\infty}$ in $X$ with $\lim_{n\rightarrow \infty }d(W(x_{n},Tx_{n},\lambda
),x_{n})=0$ and $\lim_{n\rightarrow \infty }d(x_{n},W(x_{n},Sx_{n},\lambda
))=0,$ there exists $z\in X$ such that

\begin{enumerate}
\item $\lim_{n\rightarrow \infty }d(W(z,T^{n}z,\lambda),x_{n})=0$

\item $\lim_{n\rightarrow \infty }d(x_{n},W(z,S^{n}z,\lambda))=0$
\end{enumerate}
\end{definition}

\begin{theorem}
\label{Limit sp} Let $(X,d,W)$ be a complete
convex metric space. Suppose that $T,S$ are mappings on $X$ as in the
Corollary \ref{application}. Then, The common fixed point problem of a pair $%
(T,S) $ has limit shadowing property.
\end{theorem}

\begin{proof}
For any $z\in X$, consider
\begin{align*}
d(W(z,T^{n}z,\lambda ),x_{n})& =d(T_{\lambda }^{n}z,x_{n}) \\
& \leq d(T_{\lambda }^{n}z,S_{\lambda }x_{n})+d(S_{\lambda }x_{n},x_{n}) \\
&\leq a[d(T_{\lambda }^{n-1}z,T_{\lambda }^{n}z)]^{\alpha }[d(x_{n},S_{\lambda
}x_{n})]^{1-\alpha }.
\end{align*}%
On taking limit as $n\rightarrow \infty $ on both sides of the above
inquality, we get $d(T_{\lambda }^{n}z,x_{n})=0$. Similarly, we have $%
\lim_{n\rightarrow \infty }d(x_{n},W(z,S^{n}z,\lambda ))=0$. Hence the given
common fixed point problem of a pair $(T,S)$ has the limit shadowing
property.
\end{proof}

\subsection{Ulam-Hyers Stability}

Let $(X,d,W)$ be a convex metric space, $%
T,S:X\rightarrow X$ and $\epsilon >0$ . A point $w^{\ast }\in X$ called an $%
\epsilon $-solution of the common fixed point problem for a pair $(T,S)$, if $w^{\ast }$
satisfies the following inequalities%
\begin{equation*}
d(w^{\ast },W(w^{\ast },Tw^{\ast }, \lambda))\leq \epsilon .
\end{equation*}
and,
\begin{equation*}
d(w^{\ast },W(w^{\ast },Sw^{\ast }, \lambda))\leq \epsilon .
\end{equation*}

Now we will give the notion of Ulam-Hyers stability in convex
metric space.

\begin{definition}\rm
Let $(X,d,W)$ be a convex metric space, $%
T,S:X\rightarrow X$ and $\epsilon >0$. The common fixed point problem of a pair $%
(T,S) $
is called Ulam-Hyers stable if and only if for each $\epsilon $-solution $%
w^{\ast }\in X$ of a pair $%
(T,S) $,  there exists a common
solution $x^{\ast }$ of the pair $%
(T,S) $ in $X$ such that
\begin{equation*}
d(x^{\ast },w^{\ast })\leq c \epsilon .
\end{equation*}
\end{definition}

\begin{theorem}
\label{UHS} Let $(X,d,W)$ be a complete convex
metric space. Suppose that $T,S$ be mappings on $X$ as in the Corollary \ref{application}. Then, the common fixed point problem of a pair $%
(T,S) $ is Ulam-Hyers stable.
\end{theorem}

\begin{proof}
Since $Fix(T)=Fix(T_{\lambda })$  and $Fix(S)=Fix(S_{\lambda }),$ it follows that the common fixed point problem for the pair $%
(T,S) $
 is equivalent to the common fixed point problem for pair $%
(T_{\lambda},S_{\lambda}). $
Let $w^{\ast }$ be $\epsilon $-solution of the common fixed point problem of a pair $%
(T,S) $ , that is,
\begin{equation*}
d(w^{\ast },W(w^{\ast },Tw^{\ast },\lambda))\leq \epsilon ,
\end{equation*}
and,
\begin{equation*}
d(w^{\ast },W(w^{\ast },Sw^{\ast },\lambda))\leq \epsilon ,
\end{equation*}
which is equivalent to,
\begin{equation*}
d(w^{\ast },T_{\lambda }w^{\ast })\leq \epsilon .
\end{equation*}%
and,
\begin{equation}  \label{newwwww1}
d(w^{\ast },S_{\lambda }w^{\ast })\leq \epsilon .
\end{equation}%
respectively.
Using (\ref{e form}) and  (\ref{newwwww1}), we get
\begin{align*}
d(x^{\ast },w^{\ast }) &= d(T_{\lambda}x^{\ast },w^{\ast }) \leq
d(T_{\lambda}x^{\ast },S_{\lambda}w^{\ast })+ d(S_{\lambda}w^{\ast },w^{\ast
} ) \\
& \leq a[d(x^{\ast },T_{\lambda}x^{\ast })]^{\alpha}[d(w^{\ast
},S_{\lambda}w^{\ast })]^{1-\alpha}+ d(S_{\lambda}w^{\ast },w^{\ast } ) \\
&\leq \epsilon
\end{align*}
Hence the common fixed point problem of a pair $%
(T,S) $ is Ulam-Hyers stable.
\end{proof}

\section{Conclusions}

\begin{enumerate}
\item We introduced a more generalized class of contractive mappings,
called $\mathfrak{R}$-enriched interpolative Kannan pair, and proved a
common fixed point result in the setup of $\mathfrak{R}$-complete convex
metric space.

\item We presented examples to support the concepts introduced and results
proved herein.

\item Moreover, we studied well-posedness, limit shadowing property and Ulam-Hyers stability for the results (Corollary \ref{application}) introduced herein.
\end{enumerate}


\begin{thebibliography}{99}
\bibitem{exp}  Abbas, M., \emph{Common fixed point results with applications in convex metric space}, Fasc. Math., \textbf{39} (2008), 5--15

\bibitem{Abbas2}  Abbas, M.  Anjum, R. and  Berinde, V., \emph{Enriched multivalued
contractions with applications to differential inclusions and dynamic
programming}, Symmetry., \textbf{13}(8) (2021), 1350
https://doi.org/10.3390/sym13081350

\bibitem{Abbas11} Abbas, M.  Anjum, R. and  Berinde, V., \emph{Equivalence of Certain
Iteration Processes Obtained by Two New Classes of Operators}, Mathematics.,%
\textbf{9}(18) (2021) 2292. https://doi.org/10.3390/math9182292.

\bibitem{Abbas111} Abbas, M.  Anjum, R. and  Iqbal, H.,  \emph{Generalized enriched cyclic contractions with application to generalized iterated function system}, Chaos, Solitons and Fractals., \textbf{154} (2022), https://doi.org/10.1016/j.chaos.2021.111591.


\bibitem{Agarwal}  Agarwal, R.P. and  Karapinar, E., \emph{Interpolative Rus-Reich-%
\'{C}iri\'{c} Type Contractions Via Simulation Functions}, An. St. Univ.
Ovidius Constanta, Ser. Mat., \textbf{27}(3) (2019)

\bibitem{AnjumNew}  Anjum, R. and  Abbas, M., \emph{Common Fixed point theorem for
modified Kannan enriched contraction pair in Banach spaces and its
Applications}, J. Filomat., \textbf{35}(8) (2021), 2485--2495.
https://doi.org/10.2298/FIL2108485A.

\bibitem{Anjum1}  Anjum, R. and  Abbas, M., \emph{Fixed point property of a nonempty
set relative to the class of friendly mappings}, RACSAM., \textbf{116}(1)
(2022) 1-10. https://doi.org/10.1007/s13398-021-01158-5.



\bibitem{Aydi}  Aydi, H.  Karapinar, E. and  Rold\'{a}n L\'{o}pez de
Hierro, A.F., \emph{w-Interpolative \'{C}iri\'{c}-Reich-Rus-Type Contractions},
Mathematics., \textbf{7}(57) (2019). https://doi.org/10.3390/math7010057.

\bibitem{Aydi2}  Aydi, H.  Chen, C.M. and  Karapinar, E., \emph{Interpolative \'{C}iri%
\'{c}-Reich-Rus types via the Branciari distance}, Mathematics., \textbf{7}(1)
 (2019), 84. https://doi.org/10.3390/math7010084

\bibitem{Banach}  Banach, S., \emph{Sur les op\'{e}rations dans les ensembles
abstraits et leurs applications aux \'{e}quations int\'{e}grales}, Fund.
Math., \textbf{3} (1922), 133--181

\bibitem{b}  Berinde, V., \emph{General constructive fixed point theorems for \'{C}%
iri\'{c} type almost contractions in metric spaces}, Carpathian Journal of
Mathematics., (2008), 10--19

\bibitem{B6}  Berinde, V. and  P\u{a}curar, M., \emph{Kannan's fixed point
approximation for solving split feasibility and variational inequality
problems}, Journal of Computational and Applied Mathematics., (2020), 377--427

\bibitem{B4} Berinde, V. and  P\u{a}curar, M., \emph{Approximating fixed points of
enriched contractions in Banach spaces}, Journal of Fixed Point Theory and
Applications., \textbf{22}(2) (2020), 1--10

\bibitem{B2} Berinde, V., \emph{Approximating fixed points of enriched nonexpansive
mappings by Krasnoselskii iteration in Hilbert spaces}, Carpathian J. Math.,
\textbf{35}(3) (2019), 293-304. http://dx.doi.org/10.37193/CJM.2019.03.04

\bibitem{main cite} Berinde, V. and  P\u{a}curar, M., \emph{Existence and
Approximation of Fixed Points of Enriched Contractions and Enriched $\varphi$%
-Contractions}, Symmetry., \textbf{13} (2021), 498.
https://doi.org/10.3390/sym13030498

\bibitem{B7} Berinde, V. and  P\u{a}curar, M., \emph{Approximating fixed points of
enriched Chatterjea contractions by Krasnoselskii iterative algorithm in
Banach spaces}, Journal of Fixed Point Theory and Applications., \textbf{23}%
(4) (2021), 1--16

\bibitem{B9} Berinde, V. and  P\u{a}curar, M., \emph{Fixed point theorems for
enriched \'{C}iri\'{c}-Reich-Rus contractions in Banach spaces and convex
metric spaces}, Carpathian J. Math., \textbf{37} (2021), 173--184

\bibitem{article.5}  D Blasi, F.S. and  Myjak, J., \emph{Sur la porosit\'{e} de
l'ensemble des contractions sans point fixe}, C.R. Acad. Sci. Paris., \textbf{%
308} (1989), 51--54.

\bibitem{Debnath}  Debnath, P. and  de La Sen, M., \emph{Set-valued interpolative
Hardy-Rogers and set-valued Reich-Rus-\'{C}iri\'{c}-type contractions in
b-metric spaces}, Mathematics., \textbf{7}(9) (2019), 849.
https://doi.org/10.3390/math7090849

\bibitem{Gaba}  Gaba, Y.U. and  Karapinar, E., \emph{A New Approach to the
Interpolative Contractions}, Axioms., \textbf{8}(4) (2019), 110.
https://doi.org/10.3390/axioms8040110

\bibitem{Gaba2}  Gaba, Y.U.  Aydi, H. and  Mlaik, N., \emph{$(\rho, \eta, \mu)$%
-Interpolative Kannan Contractions I}, Axioms., \textbf{10}(3) (2021), 212.
https://doi.org/10.3390/axioms10030212

\bibitem{ortho} Gordji, M.E.  Eshaghi, M.  Ramezani, M. M.  De La Sen and  Cho, Y. J., \emph{On orthogonal sets and Banach fixed point theorem}, Fixed Point
Theory., \textbf{18}(2) (2017), 569--578


\bibitem{Kannan1}  Kannan, R., \emph{Some results on fixed points}, Bull. Calcutta
Math. Soc., \textbf{60} (1968), 71--76

\bibitem{Kannan}  Kannan, R., \emph{Some results on fixed points}, II, Amer. Math.
Monthky., \textbf{76} (1969), 405--408

\bibitem{Karapinar}  Karapinar, E., \emph{Revisiting the Kannan Type Contractions
via Interpolation}, Adv. Theory Nonlinear Anal. Appl., \textbf{2}(2) (2018),
85--87. http://dx.doi.org/10.31197/atnaa.431135

\bibitem{Karapinar2}  Karapinar, E., O. Alqahtani and H. Aydi, \emph{On
Interpolative Hardy-Rogers Type Contractions}, Symmetry., \textbf{11}(1)
(2019), 8. https://doi.org/10.3390/sym11010008

\bibitem{Karapinar3}  Karapinar, E.  Agarwal R. and  Aydi, H., \emph{Interpolative
Reich-Rus-\'Ciri\'c Type Contractions on Partial Metric Spaces}, Mathematics.,
\textbf{6}(11) (2018), 256. https://doi.org/10.3390/math6110256

\bibitem{Karapinar4}  Karapinar E. and  Fulga, A., \emph{New Hybrid Contractions on
b-Metric Spaces}, Mathematics., \textbf{7}(7) (2019), 578.
https://doi.org/10.3390/math7070578

\bibitem{Kikkawa}  Kikkawa M. and  Suzuki, T., \emph{Some similarity between
contractions and Kannan mappings II}, Bull. Kyushu Inst. Technol. Pure Appl.
Math., \textbf{55} (2008), 1--13

\bibitem{R metric} Khalehoghli, S.  Rahimi, H. and
Gordji, M. E., \emph{Fixed point theorems in R-metric spaces with applications}, AIMS
Mathematics., \textbf{5}(4) (2020), 3125--3137



\bibitem{nieto} ]  Nieto, J.J. and    L\'{o}pez, R.R., \emph{Contractive mapping theorems in
partially ordered sets and applications to ordinary differential equations},
Order., \textbf{22} (2005), 223--239


\bibitem{N}  Noorwali, M., \emph{Common Fixed Point for Kannan type via
Interpolation}, J. Math. Anal., \textbf{9} (2018), 92--94

\bibitem{ran}  Ran, A.C.M. and  Reurings, M.C.B., \emph{A fixed point theorem in
partially ordered sets and some application to matrix equations}, Proc. Amer.
Math.Soc., \textbf{132} (2004), 1435--1443


\bibitem{article.11}  Reich, S. and  Zaslavski, A.J., \emph{Well-posedness of fixed
point problems}, Far East Journal of Mathematical Sciences (FJMS). Special
volume, part III, (2001), 393--401

\bibitem{Takahashi}  Takahashi, W.A., \emph{Convexity in metric space and
nonexpansive mappings}, I.Kodai Math. Sem. Rep.,\textbf{22} (1970),
142--149





\end{thebibliography}
\end{document}